\documentclass[11pt]{amsart}
\usepackage{txfonts,verbatim}
\usepackage{amscd,amssymb,amsmath,amsbsy,amsthm,amsfonts}
\usepackage[all]{xy}
\usepackage[colorlinks,plainpages,backref,urlcolor=blue]{hyperref}

\newtheorem{theorem}{Theorem}[section]
\newtheorem{lemma}[theorem]{Lemma}

\newtheorem{corollary}[theorem]{Corollary}
\newtheorem{prop}[theorem]{Proposition}

\theoremstyle{definition}
\newtheorem{definition}[theorem]{Definition}
\newtheorem{example}[theorem]{Example}
\newtheorem{remark}[theorem]{Remark}

\newtheorem{thm}{Theorem}

\newcommand{\N}{\mathbb{N}}
\newcommand{\Z}{\mathbb{Z}}
\newcommand{\Q}{\mathbb{Q}}

\newcommand{\C}{\mathbb{C}}

\newcommand{\VV}{\mathcal{V}}

\newcommand{\HH}{\mathcal{H}}

\newcommand{\cE}{\mathcal{E}}
\newcommand{\cF}{\mathcal{F}}
\newcommand{\cL}{\mathcal{L}}

\DeclareMathOperator{\rank}{rank} 
\DeclareMathOperator{\im}{im} \DeclareMathOperator{\coker}{coker}
\DeclareMathOperator{\codim}{codim}
\DeclareMathOperator{\diag}{diag}
 \DeclareMathOperator{\id}{id}

\DeclareMathOperator{\T}{Tors} \DeclareMathOperator{\Hom}{{Hom}}
 \DeclareMathOperator{\Ext}{{Ext}}

 \DeclareMathOperator{\spn}{span}

 \DeclareMathOperator{\Lie}{Lie}
\DeclareMathOperator{\Spec}{{Spec}}
\DeclareMathOperator{\Spm}{{maxSpec}}

\DeclareMathOperator{\ii}{i} \DeclareMathOperator{\ord}{ord}
 
\DeclareMathOperator{\Tors}{Tors}

\DeclareMathOperator{\alg}{{alg}}
\DeclareMathOperator{\group}{{group}}
\DeclareMathOperator{\corank}{{corank}}

\DeclareMathOperator{\Ob}{{Ob}}
\DeclareMathOperator{\Mor}{{Mor}}

\newcommand{\AbRed}{\mathsf{AbRed}}
\newcommand{\SubAbRed}{\mathsf{SubAbRed}}
\newcommand{\AbFgGp}{\mathsf{AbFgGp}}
\newcommand{\SubAbGp}{\mathsf{SubAbGp}}

\newcommand{\wG}{\widehat{G}}

\newcommand{\same}{\Longleftrightarrow}
\newcommand{\surj}{\twoheadrightarrow}
\newcommand{\inj}{\hookrightarrow}

\newcommand{\abs}[1]{\left| #1 \right|}

\def\set#1{{\{ #1\}}}

\title[Intersections of translated algebraic subtori]%
{Intersections of translated algebraic subtori}

\author[A.~Suciu]{Alexander~I.~Suciu$^1$}
\address{Department of Mathematics,
Northeastern University, Boston, MA 02115, USA}
\email{a.suciu@neu.edu}
\thanks{$^1$Partially supported by NSA grant H98230-09-1-0021
and NSF grant DMS--1010298}

\author[Y.~Yang]{Yaping~Yang}
\address{Department of Mathematics,
Northeastern University, Boston, MA 02115, USA}
\email{yang.yap@husky.neu.edu}

\author[G.~Zhao]{Gufang~Zhao}
\address{Department of Mathematics,
Northeastern University, Boston, MA 02115, USA}
\email{zhao.g@husky.neu.edu}



\begin{document}
\begin{abstract}
We exploit the classical correspondence
between finitely generated abelian groups and abelian complex
algebraic reductive groups to study the intersection theory of
translated subgroups in an abelian complex algebraic reductive group,
with special emphasis on intersections of (torsion) translated subtori in an
algebraic torus.

\end{abstract}

\maketitle
\tableofcontents

\section{Introduction}
\label{sect:intro}

\subsection{Motivation}
\label{subsec:intro0}

In this note, we study the intersection theory of translated
subtori in a complex algebraic torus, and, more generally,
of translated subgroups in an abelian complex algebraic
reductive group.  The motivation for this study comes in
large part from the investigation of characteristic varieties
and homological finiteness properties of abelian covers,
embarked upon in \cite{Su} and \cite{SYZ}.

As shown by Arapura \cite{Ar}, the jump loci for cohomology with
coefficients in rank $1$ local systems on a connected,
smooth, quasi-projective variety $X$ consist of translated
subtori of the character torus of $\pi_1(X)$. Understanding
the way these subtori intersect gives valuable information
on the Betti numbers of regular, abelian covers of $X$, 
see for instance \cite{Su, SYZ}.

Studying the intersection theory of arbitrary subvarieties
in a complex algebraic torus is beyond the scope of this work.
Nevertheless, the torsion points of the intersection of such subvarieties
can be located by considering the intersection of suitable translated
subtori.  Indeed, as shown by M.~Laurent in \cite{La}, given any
subvariety $V \subset (\C^*)^r$, there exist torsion-translated
subtori $Q_1, \dots , Q_s$ in $(\C^*)^r$ such that
$\Tors(V)=\bigcup_{j} \Tors(Q_j)$.  Consequently,
in order to locate the torsion points on an arbitrary intersection
of subvarieties, $V_1 \cap \dots \cap V_k$, it is enough to find the
torsion-translated subtori $Q_{i,j}$ corresponding to each $V_i$,
and then taking the torsion points on the variety
$\bigcap_{i} (\bigcup_{j} Q_{i,j})$.

\subsection{Pontrjagin duality}
\label{subsec:intro1}

We start by formalizing the correspondence between the category of
abelian complex algebraic reductive groups and the category of
finitely generated abelian groups.  This well-known correspondence
(sometimes called Pontrjagin duality) is based on the functor
$\Hom(-,\C^*)$.
\par
For a fixed algebraic group
$T:=\widehat{H}=\Hom_{\group}(H,\C^*)$, there is a duality
\begin{equation}
\label{eq:pont dual}
\xymatrix{
*\txt{Algebraic subgroups of $T$}
\ar@/^1pc/@<1ex>[r]|-{\ \varepsilon\ } & *\txt{Subgroups of
$H$}\ar@/^1pc/@<1ex>[l]|-{\ V\ }}
\end{equation}
where $\varepsilon$ sends  $W \subseteq T$
to $\ker\, (\Hom_{\alg}(T,\C^*) \surj \Hom_{\alg}(W,\C^*))$,
while $V$ sends  $\xi \le H$ to $\Hom_{\group}(H/\xi,\C^*)$.
Both sides of \eqref{eq:pont dual} are partially ordered sets, with naturally
defined meets and joins. As we show in Theorem \ref{theorem-isomorphic},
the above correspondence is an order-reversing equivalence of lattices.

For any subgroup $\xi$ of $H$, let
$\overline{\xi}:=\{ x\in H \mid nx\in \xi \hbox{ for some } n\in\N\}$;
the subgroup $\xi$ is called primitive if $\overline{\xi}=\xi$.
Under the correspondence $H \leftrightsquigarrow T$,
primitive subgroups of $H$ correspond to connected
algebraic subgroups of $T$.    In general, the components
of $V(\xi)$ are indexed by the ``determinant group,"
$\overline{\xi}/\xi$, while the identity component
is $V(\overline\xi)$.

\subsection{Intersections of translated subgroups}
\label{subsec:intro2}

Building on the approach taken by E.~Hironaka in \cite{Hi},
we use Pontrjagin duality to study intersections of
translated subtori in a complex algebraic torus, and, more
generally, translated subgroups in an abelian algebraic
reductive group over $\C$.  This allows us to decide
whether a finite collection of translated subgroups
intersect non-trivially, and, if so, what the dimension of their
intersection is.

More precisely, let $\xi_1, \dots, \xi_k$ be subgroups of $H$,
and let $\xi$ be their sum.
Let $\sigma\colon \xi_1\times \cdots \times \xi_k \to
\xi$ be the sum homomorphism, and let
$\gamma\colon \xi_1\times \cdots \times \xi_k \to
H^k$ be the product of the inclusion maps.
Finally, let $\eta_1, \dots, \eta_k$ be elements in $T=\widehat{H}$,
and $\eta = (\eta_1, \dots, \eta_k)\in T^{k}$.

\begin{thm}[Theorem~\ref{translated-prop}]
\label{thm:intro1}
The intersection $Q=\eta_1 V(\xi_1) \cap \cdots \cap \eta_k V(\xi_k)$
is non-empty if and only if $\hat\gamma(\eta) \in \im(\hat\sigma)$,
in which case $Q=\rho V(\xi)$, for some $\rho\in Q$.
Furthermore, if the intersection is non-empty, then
\begin{enumerate}
\item  \label{rr1}
$Q$ decomposes into
irreducible components as
$Q= \bigcup_{\tau \in \widehat{\overline{\xi}/\xi}}
\rho \tau V(\overline{\xi})$, and $\dim(Q)=\dim(T)-\rank(\xi)$.

\item \label{rr2}
If $\eta$ has finite order, then $\rho$ can be chosen
to have finite order, too.  Moreover, $\ord(\eta) \mid
\ord(\rho) \mid c\cdot \ord(\eta) $,
where $c$ is the largest order of
any element in $\overline{\xi}/\xi$.
\end{enumerate}
\end{thm}

As a corollary, we obtain a general description of the
intersection of two arbitrary unions of translated subgroups,
$W=\bigcup_i \eta_i V(\xi_i)$ and $W'=\bigcup_j \eta_j' V(\xi_j')$:
\begin{equation}
\label{eq:int}
W \cap W' = \bigcup_{i, j:\: \hat\gamma_{i, j}(\eta_i, \eta_j)
\in\im(\hat\sigma_{i, j})} \eta_{i, j} V(\xi_i+\xi_j').
\end{equation}
Moreover, if all the elements $\eta_i$ and $\eta'_j$
have finite order, so do the elements $\eta_{i,j}$.

In the case when $H= \Z^r$ and $T=(\C^{*})^{r}$,
Theorem \ref{thm:intro1} recovers a result from \cite{Hi}.

\subsection{Exponential interpretation}
\label{subsec:intro3}

Next, we turn to the relationship between the correspondence
$H \leftrightsquigarrow T$ and the exponential map $\exp\colon \Lie(T)\to T$.

Setting $\HH=H^{\vee}:=\Hom(H, \Z)$, the exponential
map may be identified with the map
$\Hom(\HH^{\vee}, \C) \to \Hom(\HH^{\vee}, \C^*)$
induced by  $\C\to \C^*$, $z\mapsto e^{2 \pi \ii z}$.
Furthermore, if $\chi\le \HH$ is a sublattice, then
$V((\HH/\chi)^{\vee})= \exp(\chi \otimes \C)$.

\begin{thm}[Theorem~\ref{thm:exp}]
\label{thm:intro2}
Let $T$ be a complex abelian reductive group, and
let $\chi_1$ and $\chi_2$ be two sublattices of 
$\HH=\Hom_{\alg}(T, \C^*)^{\vee}$.
\begin{enumerate}
\item \label{b1}
Set $\xi=(\HH/\chi_1)^{\vee} +(\HH/\chi_2)^{\vee}$.
We then have an equality of algebraic groups,
\[
\exp(\chi_1 \otimes \C) \cap \exp(\chi_2 \otimes \C) =
\widehat{\overline{\xi}/\xi}\cdot V(\overline{\xi}).
\]

\item \label{b2}
Now suppose  $\chi_1$ and $\chi_2$ are
primitive sublattices of $\HH$, with $\chi_1 \cap \chi_2=0$.
We then have
an isomorphism of finite abelian groups,
\[
\exp(\chi_1 \otimes \C) \cap \exp(\chi_2 \otimes \C) \cong
\overline{\chi_1 + \chi_2}/\chi_1 + \chi_2.
\]
\end{enumerate}
\end{thm}

In the case when $T=(\C^*)^r$, part \eqref{b2} recovers (via a
different approach) the main result from \cite{Na}.

\subsection{Intersections of torsion-translated subtori}
\label{subsec:intro4}

In the case when all the translation factors $\eta_i$
of the subtori $V(\xi_i)$ appearing in Theorem \ref{thm:intro1}
are of finite order, we can say more about the intersection
$\bigcap_{i=1}^{k} \eta_i V(\xi_i)$.  Since intersections of
torsion-translated subtori are again torsion-translated subtori,
we may assume $k=2$.

Let $\xi$ be a primitive lattice in $\Z^r$.  Given a vector
$\lambda\in \Q^r$, we say $\lambda$ virtually belongs to 
$\xi$ if $d\cdot \lambda\in \Z^r$, where $d$ is the determinant
of the matrix $[\xi \mid \xi_0]$ obtained by concatenating
basis vectors for the sublattices $\xi$
and $\xi_0=\Q\lambda \cap \Z^r$.

\begin{thm}[Theorem~\ref{thm:k=2}]
\label{thm:intro3}
Let $\xi_1$ and $\xi_2$ be two sublattices in $\Z^r$.
Set $\varepsilon=\xi_1\cap \xi_2$, and write
$\widehat{\overline{\varepsilon}/\varepsilon}=\{\exp(2\pi \ii \mu_k) \}_{k=1}^s$.
Also let $\eta_1$ and $\eta_2$ be two torsion elements in $(\C^*)^r$,
and write $\eta_j=\exp(2\pi \ii \lambda_j)$. The following are equivalent:
\begin{enumerate}
\item  \label{qq1}
The variety $Q=\eta_1 V(\xi_1) \cap \eta_2 V(\xi_2)$ is non-empty.
\item \label{qq2}
One of the vectors $\lambda_1-\lambda_2 - \mu_k$ virtually belongs to
the lattice $(\Z^r/\varepsilon)^{\vee}$.
\end{enumerate}
If either condition is satisfied, then $Q=\rho V(\xi_1+\xi_2)$
for some $\rho \in Q$.
\end{thm}

This theorem provides an algorithm for checking the condition from
Theorem \ref{thm:intro1}, solely in terms of arithmetic data
extracted from the lattices $\xi_1$ and $\xi_2$ and the rational
vectors  $\lambda_1$ and $\lambda_2$.  The complexity of this
algorithm is linear with respect to the order of the determinant group
$\overline{\varepsilon}/\varepsilon$.

\subsection{Some applications}
\label{subsec:intro5}

As mentioned in \S\ref{subsec:intro0}, one of our main motivations
comes from the study of characteristic varieties, especially
as it regards the intersection pattern of the components
of such varieties, and the count of their torsion points.
Theorem \ref{thm:intro1} yields several corollaries, which
provide very specific information in this direction; this information
has since been put to use in \cite{Su, SYZ}.

In addition to these applications, we also
consider the following counting problem: how
many translated subtori does an algebraic torus
have, once we fix the identity component,
and the number of irreducible components?
Using the correspondence from \eqref{eq:pont dual},
and a classical result on the number of finite-index
subgroups of a free abelian group,
we express the generating function for this counting
problem in terms of the Riemann zeta function.

\section{Finitely generated abelian groups and abelian reductive groups}
\label{sect:corresp}

In this section, we describe an order-reversing isomorphism
between the lattice of subgroups of a finitely generated
abelian group $H$ and the lattice of algebraic
subgroups of the corresponding abelian, reductive,
complex algebraic group $T$.

\subsection{Abelian reductive groups}
\label{subsec:abelredgp}

We start by recalling a well-known equivalence between two categories:
that of finitely generated abelian groups, $\AbFgGp$, and that
of abelian, reductive, complex algebraic groups, $\AbRed$.
For a somewhat similar approach, see also \cite{Hi} and
\cite{Mi}.

Let $\C^*$ be the multiplicative group of units in the field $\C$
of complex numbers.
Given a group $G$, let $\wG=\Hom(G, \C^{*})$ be the group of
complex-valued characters of $G$, with pointwise multiplication
inherited from $\C^*$, and identity the character taking constant
value $1 \in \C^*$ for all $g \in G$. If the group $G$ is finitely generated,
then $\wG$ is an abelian, complex reductive algebraic group.

Note that $\wG\cong \widehat{H}$, where $H$ is the maximal abelian
quotient of $G$.  If $H$ is torsion-free, say, $H=\Z^r$, then $\widehat{H}$
can be identified with the complex algebraic torus $(\C^{*})^{r}$.
If $A$ is a finite abelian group, then $\widehat{A}$ is, in fact,
isomorphic to $A$.

Given a homomorphism $\phi \colon G_1\to G_2$, let $\hat{\phi}\colon
\widehat{G}_2 \to \widehat{G}_1$ be the induced morphism between
character groups, given by $\hat\phi(\rho)=\rho\circ \phi$.
Since $\C^*$ is a divisible abelian group, the functor
$H\leadsto \widehat{H}=\Hom(H,\C^*)$ is exact.

Now let $T$ be an abelian, complex algebraic reductive group.
We can then associate to $T$ its weight group,
$\check{T}=\Hom_{\alg}(T,\C^*)$,
where the hom set is taken in the category of algebraic groups.
It turns out that $\check{T}$ is a finitely generated abelian group,
which can be described concretely, as follows.

According to the classification of abelian reductive groups over $\C$
(cf.~\cite{Spr}), the identity component of $T$ is an algebraic
torus, i.e., it is of the form $(\C^{*})^{r}$ for some integer $r\ge 0$.
Furthermore, this identity component has to be a normal subgroup.
Thus, the algebraic group $T$ is isomorphic to $(\C^{*})^{r}\times A$,
for some finite abelian group $A$.

The coordinate ring $\mathcal{O}[T]$ decomposes as
\begin{equation}
\label{eq:coord ring}
\mathcal{O}[(\C^{*})^{r}\times A]\cong\mathcal{O}[(\C^{*})^{r}]
\otimes\mathcal{O}[A]\cong\C[\Z^r]\otimes\C[\widehat{A}],
\end{equation}
where $\C[G]$ denotes the group ring of a group $G$.
Let $\mathcal{O}[T]^*$ be the group of units in the coordinate
ring of $T$.  By \eqref{eq:coord ring}, this group is isomorphic to
$\C^*\times \Z^r\times A$, where $\C^*$ corresponds
to the non-zero constant functions.  Taking the quotient
by this $\C^*$ factor, we get isomorphisms
\begin{equation}
\label{eq:checkt}
\check{T} \cong \mathcal{O}[T]^*/\C^* \cong \Z^r\times A.
\end{equation}
Clearly, $\Spm\, (\C[\check{T}]) = \Hom_{\alg}(\C[\check{T}], \C)
= \Hom_{\group}(\check{T}, \C^*)= T$.

Now let  $f\colon T_1\to T_2$ be a morphism in $\AbRed$.
Then the induced morphism on coordinate rings,
$f^*\colon  \mathcal{O}[T_2] \to \mathcal{O}[T_1]$,
restricts to a group homomorphism,
$f^*\colon  \mathcal{O}[T_2]^* \to \mathcal{O}[T_1]^*$,
which takes constants to constants.
Under the identification from \eqref{eq:checkt}, $f^*$ induces
a homomorphism $\check{f}\colon \check{T}_2\to \check{T}_1$
between the corresponding weight groups.

The following proposition is now easy to check.

\begin{prop}
\label{prop:abred}
The functors $H\leadsto \widehat{H}$ and $T\leadsto \check{T}$
establish a contravariant equivalence between the category
of finitely generated abelian groups and the category of
abelian reductive groups over $\C$.
\end{prop}

Recall now that the functor  $H\leadsto \widehat{H}$ is exact.
Hence, the functor $T\leadsto  \check{T}$ is also exact.

\begin{remark}
\label{rem:prods}
The above functors behave well with respect to (finite) direct products.
For instance, let $\alpha\colon A\to C$ and $\beta\colon B\to C$ be two
homomorphisms between finitely generated abelian groups,
and consider the homomorphism $\delta\colon A\times B\to C$
defined by $\delta(a,b)=\alpha(a)+\beta(b)$.
The morphism $\hat\delta\colon \widehat{C}\to \widehat{A\times B}=
\widehat{A}\times \widehat{B}$ is then given by
$\hat\delta(f)= (\hat{\alpha}(f),\hat{\beta}(f))$.
\end{remark}

\subsection{The lattice of subgroups of a finitely generated abelian group}
\label{subsec:lattice H}

Recall that a poset $(L, \leq)$ is a {\em lattice}\/ if every pair of elements
has a least upper bound and a greatest lower bound.  Define operations
$\vee$ and $\wedge$ on $L$ (called join and meet, respectively) by
$x \vee y = \text{sup}\{x,y\}$ and $x \wedge y = \text{inf}\{x,y\}$.

The lattice $L$ is called {\em modular}\/
if, whenever $x < z$, then $x \vee (y \wedge z) = (x \vee y) \wedge z$,
for all $y\in L$. The lattice $L$ is called {\em distributive}\/ if
$x \vee (y \wedge z) = (x \vee y) \wedge (x \vee z)$
and $x \wedge (y \vee z) = (x \wedge y) \vee (x \wedge z)$,
for all $x, y, z \in L$.   A lattice is modular if and only if it
does not contain the pentagon as a sublattice, whereas
a modular lattice is distributive if and only if it does not
contain the diamond as a sublattice.

Finally, $L$ is said to be a {\em ranked lattice}\/ if there is a function
$r\colon L\to \Z$ such that $r$ is constant on all minimal elements,
$r$ is monotone (if $\xi_1 \le \xi_2$ , then $r(\xi_1) \leq r(\xi_2)$),
and $r$ preserves covering relations (if  $\xi_1 \le \xi_2$, but there
is no $\xi$ such that $\xi_1 < \xi < \xi_2$, then $r(\xi_2)=r(\xi_1)+1$).

Given a group $G$, the set of subgroups of $G$ forms a lattice,
$\cL(G)$, with order relation given by inclusion. The join of two subgroups,
$\gamma_1$ and $\gamma_2$, is the subgroup generated by
$\gamma_1$ and $\gamma_2$, and their meet is the intersection
$\gamma_1 \cap \gamma_2$.  There is unique minimal element---%
the trivial subgroup, and a unique maximal element---%
the group $G$ itself. The lattice $\cL(G)$ is distributive
if and only if $G$ is locally cyclic (i.e., every finitely generated
subgroup is cyclic).  Similarly, one may define the lattice of
normal subgroups of $G$; this lattice is always modular.
We refer to \cite{Sch} for more on all this.

Now let $H$ be a finitely generated abelian group,
and let $\cL(H)$ be its lattice of subgroups. In this case,
the join of two subgroup, $\xi_1$ and $\xi_2$, equals
the sum $\xi_1 + \xi_2$.  By the above, the lattice $\cL(H)$
is always modular, but it is not a distributive lattice, unless
$H$ is cyclic.  Furthermore, $\cL(H)$ is a ranked lattice, with
rank function $\xi\mapsto \rank(\xi)=\dim_{\Q}(\xi\otimes \Q)$
enjoying the following property:
$\rank(\xi_1)+\rank(\xi_2)=
\rank(\xi_1 \wedge \xi_2)+\rank(\xi_1 \vee \xi_2)$.

\subsection{The lattice of algebraic subgroups of a complex algebraic torus}
\label{subsec:lattice T}

Now let $T$ be a complex abelian reductive group,
and let $(\cL_{\alg}(T),\le)$ be the poset of algebraic
subgroups of $T$, ordered by inclusion.
It is readily seen that $\cL_{\alg}(T)$ is a ranked modular lattice.
For any algebraic subgroups $P_1$ and $P_2$ of $T$,
the join $P_1 \vee P_2=P_1\cdot P_2$ is the algebraic subgroup
generated by the two subgroups $P_1$ and $P_2$, while
the meet $P_1 \wedge P_2=P_1 \cap P_2$ is the intersection
of the two subgroups.  Furthermore, the rank of a subgroup
$P$ is its dimension.

The next theorem shows that the natural correspondence
from Proposition \ref{prop:abred} is lattice-preserving.
Recall that, if $H$ is a finitely generated abelian group,
the character group $\widehat{H}=\Hom_{\group}(H,\C^*)$
is an abelian reductive group over $\C$, and conversely,
if $T$ is an abelian reductive group, the weight group
$\check{T}=\Hom_{\alg}(T,\C^*)$ is a finitely generated
abelian group.

\begin{theorem}
\label{theorem-isomorphic}
Suppose $H\cong\check{T}$, or equivalently, $T\cong\widehat{H}$.
There is then an order-reversing isomorphism between the
lattice of subgroups of $H$ and the lattice of algebraic
subgroups of $T$.
\end{theorem}

\begin{proof}
For any subgroup $\xi \le H$, let
\begin{equation}
\label{eq:ve}
V(\xi)=\Spm (\C[H /\xi])
\end{equation}
be the set of closed points of $\Spec(\C[H /\xi])$.
Clearly, the variety $V(\xi)$ embeds into $\Spm(\C[H]) \cong T$
as an algebraic subgroup.  This subgroup can be
naturally identified with the group of characters of
$H/\xi$, that is, $V(\xi)\cong \widehat{H/\xi}$, or equivalently,
\begin{equation}
\label{eq:xitv}
\widehat{\xi}\cong T/V(\xi).
\end{equation}

For any algebraic subset $W \subset T$, define
\begin{equation}
\label{eq:ev}
\epsilon(W)=\ker\, (\Hom_{\alg}(T,\C^*) \surj \Hom_{\alg}(W,\C^*)).
\end{equation}
Write $T=(\C^{*})^r \times \Z_{k_1} \times \cdots \times \Z_{k_s}$,
where $\Z_{k_i}$ embeds in $\C^*$ as the subgroup of $k_i$-th
roots of unity.  Using the standard coordinates of $(\C^{*})^{r+s}$,
we can identify $\epsilon(W)$ with the subgroup $\{ \lambda \in H :
\text{$t^{\lambda}-1$ vanishes on $W$}\}$.

The proof of the theorem is completed by the next three lemmas.
\end{proof}

\begin{lemma}
\label{lem:eve} If $\xi$ is a subgroup of $H$, then
$\epsilon(V(\xi))=\xi$.
\end{lemma}

\begin{proof}
The inclusion $\epsilon(V(\xi)) \supseteq \xi$ is clear.
Now suppose $\lambda \in H \setminus \xi$. Then we may
define a character $\rho \in \widehat{H}$ such that
$\rho(\lambda) \neq 1$, but $\rho(\xi)=1$.
Evidently, $\lambda \not\in \epsilon(V(\xi))$, and we are done.
\end{proof}

Given two algebraic subgroups, $P$ and $Q$ of $T$, let
$\Hom_{\alg}(P, Q)$ be the set of morphisms from $P$ to
$Q$ which preserve both the algebraic and multiplicative structure.

\begin{lemma}
\label{lem:pv}
Any algebraic subgroup $P$ of $T$ is of the
form $V(\xi)$, for some subgroup $\xi \subseteq H$.
\end{lemma}

\begin{proof}
Let $\iota\colon P \inj T$ be the inclusion map.  Applying
the functor $\Hom_{\alg}(-, \C^*)$, we obtain an epimorphism
$\iota^*\colon H \surj \Hom_{\rm alg}(P, \C^*)$.
Set $\xi=\ker(\iota^*)$;  then $P=V(\xi)$.
\end{proof}

The above two lemmas (which generalize Lemmas 3.1 and
3.2 from \cite{Hi}), show that we have a natural correspondence
between algebraic subgroups of $T$ and subgroups of $H$.
This correspondence preserves the lattice structure on both
sides. That is, if $\xi_1\le \xi_2$, then $V(\xi_1)\supseteq V(\xi_2)$,
and similarly,  if $P_1\subseteq P_2$, then
$\epsilon(P_1)\ge \epsilon(P_2)$.

\begin{lemma}
\label{lem:vlattice}
The natural correspondence between algebraic
subgroups of $T$ and subgroups of $H$ is an
order-reversing lattice isomorphism. In particular,
\[
V(\xi_1 + \xi_2)=V(\xi_1) \cap V(\xi_2)
\quad\text{and}\quad V(\xi_1 \cap \xi_2)=V(\xi_1)\cdot
V(\xi_2).
\]
\end{lemma}

\begin{proof}
The first claim follows from the two lemmas above.
The two equalities are consequences of this.
\end{proof}

The above correspondence reverses ranks, i.e.,
$\rank(\xi) = \codim V(\xi)$ and $\corank(\xi) = \dim V(\xi)$.

\subsection{Counting algebraic subtori}
\label{subsec:count}

As a quick application, we obtain a counting formula
for the number of algebraic subgroups of an $r$-dimensional
complex algebraic torus, having precisely $k$ connected
components, and a fixed, $n$-dimensional subtorus as
the identity component.  It is convenient to restate such
a problem in terms of a zeta function.

\begin{definition}
\label{def:zeta}
Let $T$ be an abelian reductive group, and let
$T_0$ be a fixed connected algebraic subgroup.
Define the zeta function of this pair as
\[
\zeta(T, T_0, s)=\sum_{k=1}^{\infty}\frac{a_k(T,T_0)}{k^s},
\]
where $a_k(T,T_0)$ is the number of algebraic subgroups
$W \leq T$ with identity component equal to $T_0$ and
such that $\abs{W/T_0}=k$.
\end{definition}

This definition is modeled on that of the zeta function of a
finitely generated group $G$, which is given by
$\zeta(G, s)=\sum_{k=1}^{\infty} a_k(G)k^{-s}$, where
$a_k(G)$ is the number of index-$k$ subgroups of $G$,
see for instance \cite{SW}.

\begin{corollary}
\label{cor:zetas}
Suppose $T \cong (\C^{*})^{r}$ and $T_0 \cong (\C^{*})^{n}$, for some
$0\leq n\leq r$.  Then $\zeta(T, T_0, s)=\zeta(s)\zeta(s-1)\cdots
\zeta(s-r+n+1)$, where $\zeta(s)=\sum_{k=1}^{\infty}k^{-s}$
is the usual Riemann zeta function.
\end{corollary}

\begin{proof}
First assume $n=0$, so that $T_0=\{1\}$. By
Theorem \ref{thm: equivalence},
$a_k(T,\{1\})$ is the number of algebraic
subgroups $W \leq T$ of the form
$W=V(\xi)$, where $\xi\le \Z^r$ and $[\Z^r : \xi]=k$.
Clearly, this number equals $a_k(\Z^r)$, and
so $\zeta(T, \{1\}, s)=\zeta(\Z^r, s)$.  By a classical result of
Bushnell and Reiner (see \cite{SW}), we have that
$\zeta(\Z^r, s)=\zeta(s)\zeta(s-1)\cdots \zeta(s-r+1)$.
This proves the claim for $n=0$.

For $n > 0$, we simply take the quotient of the group $T$
by the fixed subtorus $T_0$, to get
$\zeta(T, T_0, s)=\zeta(T/T_0, \{1\}, s)=\zeta(\Z^{r-n}, s)$.
This ends the proof.
\end{proof}

\section{Primitive lattices and connected subgroups}
\label{sect:primitive}

In this section, we show that the correspondence between $\cL(H)$
and  $\cL_{\alg}(T)$ restricts to a correspondence between
the primitive subgroups of $H$ and the connected algebraic
subgroups of $T$.  We also explore the relationship between
the connected components of $V(\xi)$ and the determinant
group, $\overline{\xi}/\xi$, of a subgroup $\xi\le H$.

\subsection{Primitive subgroups}
\label{subsec:primitive}
As before, let $H$ be a finitely generated abelian group.
We say that a subgroup $\xi \le H$ is {\em primitive}\/
if there is no other subgroup $\xi'\le H$ with $\xi< \xi'$ and
$[\xi' : \xi] <\infty$.  In particular, a primitive subgroup
must contain the torsion subgroup, $\Tors(H)=\{\lambda\in H \mid
\exists\, n \in \N \hbox{ such that } n\lambda =0\}$.

The intersection of two primitive subgroups is again
a primitive subgroup, but the sum of two primitive
subgroups need not be primitive (see, for instance,
the proof of Corollary \ref{cor:dimtori}).   Thus, the set
of primitive subgroups of $H$ is not necessarily
a sublattice of $\cL(H)$.

When $H$ is free abelian,
then all subgroups $\xi \le H$ are also free abelian.  In this case,
$\xi$ is primitive if and only if it has a basis that can be extended to
a basis of $H$, or, equivalently, $H/\xi$ is torsion-free.
It is customary to call such a subgroup a {\em primitive lattice}.

Returning to the general situation, let $H$ be a finitely generated
abelian group.  Given an arbitrary subgroup $\xi \le H$,
define its primitive closure, $\overline{\xi}$, to be the
smallest primitive subgroup of $H$ containing $\xi$.  Clearly,
\begin{equation}
\label{eq:bar epsilon}
\overline{\xi}=\{\lambda \in H: \exists\,
n \in \N \hbox{ such that } n\lambda \in \xi\}.
\end{equation}

Note that $H/\overline{\xi}$ is torsion-free, and thus we have
a split exact sequence, \\[-5pt]
\begin{equation}
\label{eq:split}
\xymatrix{0 \ar[r] & \overline{\xi}  \ar[r] & H
\ar[r]&H/\overline{\xi}\ar@/_15pt/@{-->}[l]   \ar[r]& 0}.
\end{equation}

By definition, $\xi$ is a finite-index subgroup of $\overline{\xi}$; in
particular, $\rank(\xi) = \rank(\overline{\xi})$.  We call the quotient
group, $\overline{\xi}/\xi$, the {\em determinant group}\/ of $\xi$.
We have an exact sequence,
\begin{equation}
\label{eq:ex}
\xymatrix{0 \ar[r] & H/\overline{\xi}  \ar[r] & H/\xi
\ar[r]&\overline{\xi}/\xi  \ar[r]& 0},
\end{equation}
with $\overline{\xi}/\xi$ finite.  The inclusion $\overline{\xi} \inj H$
induces a monomorphism $\overline{\xi}/\xi \inj H/\xi$,
which yields a splitting for the above sequence,
showing that $\overline{\xi}/\xi \cong \Tors(H/\xi)$.
Since the group $\overline{\xi}/\xi$ is finite, it is
isomorphic to its character group,
$\widehat{\overline{\xi}/\xi}$, which in turn can be
viewed as a (finite) subgroup of $\widehat{H}=T$.

Using an approach similar
to the one from \cite[Lemma 3.3]{Hi}, we sharpen
and generalize that result, as follows.

\begin{lemma}
\label{lem:vxi}
For every subgroup $\xi \le H$, we have an isomorphism
of algebraic groups,
\begin{equation}
\label{eq:vdecomp}
V(\xi) \cong \widehat{\overline{\xi}/\xi} \cdot V(\overline{\xi}).
\end{equation}
Moreover,
\begin{enumerate}
\item \label{vx1}
$V(\xi)=\bigcup_{\rho \in \widehat{\overline{\xi}/\xi}}\rho V(\overline{\xi})$
is the decomposition of $V(\xi)$ into irreducible components, with
$V(\overline{\xi})$ as the component of the identity.

\item  \label{vx2}
$V(\xi)/V(\overline{\xi}) \cong \widehat{\overline{\xi}/\xi}$.  In particular,
if $\rank \xi= \dim H$, then $V(\xi) \cong \widehat{\overline{\xi}/\xi}$.
\end{enumerate}
\end{lemma}

\begin{proof}
As noted in \S\ref{subsec:abelredgp}, the functor $\Hom(-,\C^{*})$ is
exact.  Applying this functor to sequence \eqref{eq:ex}, we obtain
an exact sequence in $\AbRed$,
\begin{equation}
\label{eq:ex2}
\xymatrixrowsep{6pt}
\xymatrix{0 \ar[r] & \widehat{\overline{\xi}/\xi}   \ar[r] & \widehat{H/\xi}
\ar@{=}[d]
\ar[r]& \widehat{H/\overline{\xi}}\ar@{=}[d]  \ar[r]& 0\\
& & V(\xi) & V(\overline{\xi})
}.
\end{equation}
Since sequence \eqref{eq:ex} is split, sequence \eqref{eq:ex2} is
also split, and thus we get decomposition \eqref{eq:vdecomp}.

Now recall that $H/\overline{\xi}$ is torsion-free; thus,
$V(\overline{\xi})=\Spm (\C[H/\overline{\xi}]) $ is a
connected algebraic subgroup of $T$.
Claims \eqref{vx1} and \eqref{vx2} readily follow.
\end{proof}

\begin{example}
\label{ex:toy}
Let $H=\Z$ and identify $\widehat{H}=\C^*$.
If $\xi=2\Z$, then $V(\xi)=\{\pm 1\} \subset \C^*$,
whereas $\overline{\xi}=H$ and
$V(\overline{\xi})=\{ 1\} \subset \C^*$.
\end{example}

Clearly, the subgroup $\xi$ is primitive if and only if $\xi=\overline{\xi}$.
Thus, $\xi$ is primitive if and only if the
variety $V(\xi)$ is connected. Putting things together, we obtain
the following corollary to Theorem \ref{theorem-isomorphic}
and Lemma \ref{lem:vxi}.

\begin{corollary}
\label{cor:primitive}
Let $H$ be a finitely generated abelian group, and let
$T=\widehat{H}$.
The natural correspondence between $\cL(H)$ and  $\cL_{\alg}(T)$
restricts to a correspondence between the primitive subgroups
of $H$ and the connected algebraic subgroups of $T$.
\end{corollary}

\subsection{The dual lattice}
\label{subsec:dual}

Given an abelian group $A$, let $A^{\vee}=\Hom(A, \Z)$
be the dual group.  Clearly, if $H$ is a finitely generated abelian group,
then $H^{\vee}$ is torsion-free, with $\rank H^{\vee}=\rank H$.

Now suppose $\xi \leq H$ is a subgroup.  By passing to duals,
the projection map $\pi\colon H\surj H/\xi$ yields a monomorphism
$\pi^{\vee}\colon (H/\xi)^{\vee}\inj H^{\vee}$.  Thus, $(H/\xi)^{\vee}$
can be viewed in a natural way as a subgroup of $H^{\vee}$.
In fact, more is true.

\begin{lemma}
\label{lem:dual lattice}
Let $H$ be a finitely generated abelian group, and
let $\xi \leq H$ be a subgroup.  Then $(H/\xi)^{\vee}$ is a
primitive lattice in $H^{\vee}$.
\end{lemma}

\begin{proof}
Dualizing the short exact sequence $0 \to \xi \to H
\xrightarrow{\pi} H/\xi \to 0$, we obtain a long exact sequence,
\begin{equation}
\label{eq:ext}
\xymatrix{0 \ar[r]& (H/\xi)^{\vee}
\ar^(.6){\pi^{\vee}}[r]& H^{\vee}  \ar[r] & \xi^{\vee}
\ar[r]&\Ext(H/\xi, \Z)  \ar[r]& 0}.
\end{equation}
Upon identifying $\Ext(H/\xi, \Z) =\T(H/\xi) = \overline{\xi}/\xi$
and setting $K=\coker(\pi^{\vee})$, the above sequence
splits into two short exact sequences, $0 \to (H/\xi)^{\vee}
\to H^{\vee} \to K \to 0$ and
\begin{equation}
\label{eq:K}
\xymatrix{0 \ar[r] & K \ar[r] &  \xi^{\vee} \ar[r] &
\overline{\xi}/\xi \ar[r] &  0}.
\end{equation}
\par
Now, since $K$ is a subgroup of $\xi^{\vee}$, is must be
torsion free. Thus, $(H/\xi)^{\vee}$ is a
primitive lattice in $H^{\vee}$.
\end{proof}

Given two primitive subgroups $\xi_1, \xi_2 \leq H$, their sum,
$\xi_1 + \xi_2$, may not be a primitive subgroup of $H$.
Likewise, although both $(H/\xi_1)^{\vee}$ and $(H/\xi_2)^{\vee}$
are primitive subgroups of $H^{\vee}$, their sum may not be
primitive.  Nevertheless, the following lemma shows that
the respective determinant groups are the same.

\begin{prop}
\label{prop:xi}
Let $H$ be a finitely generated abelian group, and
let $\xi_1$ and $\xi_2$ be primitive subgroups of $H$,
with $\xi_1 \cap \xi_2$ finite. Set $\xi=(H/\xi_1)^{\vee} +
(H/\xi_2)^{\vee}$, and let $\overline{\xi}$ be the primitive closure
of $\xi$ in $H^{\vee}$.  Then
\[
\overline{\xi}/\xi \cong\overline{\xi_1 + \xi_2}/\xi_1 + \xi_2.
\]
\end{prop}

\begin{proof}
Replacing $H$ by $H/\Tors(H)$ and $\xi_i$ by $\xi_i/\Tors(H)$
if necessary, we may assume that $H$ is free abelian, and $\xi_1$
and $\xi_2$ are primitive lattices with $\xi_1\cap \xi_2=\{0\}$.
Furthermore, choosing a splitting of $H^{\vee} /\overline{\xi}\inj H^{\vee}$
if necessary, we may assume that $\overline{\xi} \cong H^{\vee}$,
or equivalently, $\overline{\xi_1 + \xi_2}=H$. Write
$s=\rank \xi_1$ and $t=\rank \xi_2$. Then $s+t=n$,
where $n=\rank H$.

Choose a basis $\{ e_1, \dots, e_s\}$ for $\xi_1$.  Since
$\xi_1$ is a primitive lattice in $H$, we may extend this basis
to a basis $\{ e_1, \dots , e_s, f_1, \dots , f_t\}$ for $H$.  Picking a
suitable basis for $\xi_2$, we may assume that the inclusion
$\iota\colon \xi_1 + \xi_2 \inj H$ is given by a matrix of the form
\begin{equation}
\label{eq:mat1}
\newcommand*{\temp}{\multicolumn{1}{r|}{}}
\left(\begin{array}{ccc}
I_s &\temp & C\\
\cline{1-3} 0 &\temp & D
\end{array}\right),
\end{equation}
where $I_s$ is the $s \times s$  identity matrix and
$D=\diag(d_1,\dots, d_t)$ is a diagonal matrix with
positive diagonal entries $d_1,\dots, d_t$ such that
$1=d_1=\cdots = d_{m-1}$ and
$1\ne d_t \mid d_{t-1} \mid \cdots \mid d_m$,
for some $1 \leq m \leq t+1$. Then
$\overline{\xi_1 + \xi_2}/\xi_1 + \xi_2 =
\Z/d_m\Z\oplus \cdots \oplus \Z/d_t\Z$.

Notice that the columns of the matrix
$\left(\begin{matrix} C \\ D \\ \end{matrix}\right)$
form a basis for $\xi_2$.  Since $\xi_2$ is a primitive lattice in $H$,
the last $t-m$ columns of $C$ must have a minor of size
$t-m$ equal to $\pm 1$.  Without loss of generality,
we may assume that the corresponding rows are
also the last $t-m$ ones.

The canonical projection $\pi\colon H \surj H/\xi_1 + H/\xi_2$
is given by a matrix of the form
\begin{equation}
\label{eq:mat2}
\newcommand*{\temp}{\multicolumn{1}{r|}{}}
\left(\begin{array}{ccc}
X &\temp & Y\\
\cline{1-3} 0 &\temp & I_t
\end{array}\right).
\end{equation}
Using row and column operations, the matrix $X$ can be brought
to the diagonal form $\diag(x_1,\dots, x_s)$, where $x_i$ are positive
integers with
$1=x_1=\dots = x_{a-1}$ and
$1\ne x_s \mid x_{s-1} \mid \cdots \mid x_a$.
Moreover, the submatrix
of $Y$ involving the last $s-a+1$ rows and columns is invertible.
Taking the dual basis of $H^{\vee}=\Hom(H, \Z)$,
the inclusion $\xi=(H/\xi_1)^{\vee}
+(H/\xi_2)^{\vee} \inj H^{\vee}$ is given
by the matrix
$
\left(\begin{smallmatrix}
X^T & 0\\
Y^T & I_t
\end{smallmatrix}\right)$.
It follows that $\overline{\xi}/\xi =
\Z/x_a\Z\oplus \cdots \oplus \Z/x_s\Z$.

Evidently, the restriction of $\pi\circ \iota$ to $\xi_2$ is the
zero map; hence, $XC =-YD$.
For a fixed integer $k\le s$, set
$\delta_k:=d_t  d_{t-1}\cdots d_{k-s+t}$
and $y_k:= x_s x_{s-1} \cdots x_k$.
Clearly, $\delta_k$ is the gcd of all
minors of size $s-k+1$ of the submatrix in $YD$
involving the last $s-a+1$ rows and columns.
Hence, $\delta_k$ is also the gcd of all
minors of size $s-k+1$ of the corresponding
submatrix of $XC$. Thus, $\delta_k=y_k$.
Hence, $d_t=x_s, \dots, d_{m}=x_{s-t+m}$, and
$1=d_{m-1}=x_{s-t+m-1}$,
which implies $s-t+m=a$. This yields the desired conclusion.
\end{proof}

\section{Categorical reformulation}
\label{sect:cat}

In this section, we reformulate Theorem~\ref{theorem-isomorphic}
using the language of categories. In order to simultaneously
consider the category of all finitely generated abelian groups,
and the lattice structure for all the subgroups of a fixed abelian
group, we need the language of fibered categories
from \cite[\S5.1]{SGA1}, which we briefly recall here.

\subsection{Fibered categories}
\label{subsec:fibcat}

Recall that a poset $(L, \leq)$ can be seen as a small category, with
objects the same as the elements of $L$, and with an arrow
$p \to p'$ in the category $L$ if and only if $p \leq p'$.

Let $\cE$ be a category. We denote by $\Ob\cE$ the objects of $\cE$,
and by $\Mor\cE$ its morphisms. \textit{A category over $\cE$} is a category
$\cF$, together with a functor $\Phi\colon \cF\to\cE$.
For $T\in \Ob\cE$, we denote by $\cF(T)$ the subcategory of
$\cF$ consisting of objects $\xi$ with $\Phi(\xi)=T$,
and morphisms $f$ with $\Phi(f)= \id_T $.
\par
\begin{definition}
\label{def:cartesian}
Let  $\Phi\colon \cF\to\cE$ be a category over $\cE$.
Let $f\colon v\to u$ be a morphism in $\cF$, and set
$(F\colon V\to U)=(\Phi(f\colon v\to u))$. Then $f$ is
said to be a {\em Cartesian morphism}\/ if for any $h\colon v'\to u$
with $\Phi(h\colon v'\to u)=F$, there exists a unique
$h'\colon v'\to v$ in $\Mor\cF(F)$ such that $h=f\circ h'$.
\end{definition}

The above definition is summarized in the following
diagram:
\begin{equation}
\label{eq:cartesian}
\xymatrixrowsep{10pt}
\xymatrix{
v' \ar@{-->}[dd]_{h'} \ar[rdd]^{h} & \\
\\
v \ar[r]^{f} & u\\
V \ar[r]^{F} & U}
\end{equation}

\begin{definition}
\label{def:fibered}
We say that $\Phi\colon \cF\to \cE$ is a {\em fibered category}\/ if for
any morphism $F\colon V\to U$ in $\cE$, and any object $u\in\Ob(\cF(U))$,
there exists a Cartesian morphism $f\colon v \to u$, with $\Phi(f)=F$.
Moreover, the composition of Cartesian morphisms
is required to be a Cartesian morphism.
\end{definition}

We say $\cF$ is \textit{a lattice over $\cE$}\/ (or less succinctly, a
category fibered in lattices over $\cE$) if $\cF$ is a fibered category,
and every $\cF(T)$ is a lattice.

\subsection{A lattice over $\AbRed$}
\label{subsec:lat fibcat}

We now construct a category fibered in lattices over the category
of abelian complex algebraic reductive groups, $\AbRed$.
Let $\SubAbRed$ be the category with objects
\begin{equation}
\label{eq:osar}
\Ob\SubAbRed = \{i\colon W\inj P \mid
\text{ $i$ is a closed immersion of algebraic subgroups}\},
\end{equation}
and morphisms between $i\colon W\to P$ and $i'\colon W'\to P'$
the set of pairs
\begin{equation}
\label{eq:msar}
\set{ (f,g) \mid \text{$f\colon W\to W'$ and $g \colon P\to P'$
such that $i'\circ f=g\circ i$} }.
\end{equation}

Projection to the target,
\begin{equation}
\label{eq:war}
\xymatrix{
W \ar@{^{(}->}[r]^{i} \ar[d]^{f} & P \ar[d]^{g}\\
W' \ar@{^{(}->}[r]^{i'} & P'\\
} \mapsto \xymatrix{
P \ar[d]^{g}\\
P'}
\end{equation}
defines a functor from  $\SubAbRed$ to $\AbRed$.
One can easily check that this functor is a category fibered in
lattices over $\AbRed$, with Cartesian morphisms obtained
by taking preimages of subtori.  More precisely,
suppose $F\colon W \to P$ is a morphism in $\AbRed$,
and $\theta \inj P$ is an object in $\Ob\SubAbRed$;
the corresponding Cartesian morphism is then
\begin{equation}
\label{eq:ftheta}
\xymatrix{
F^{-1}(\theta) \ar@{^{(}->}[d]^{i} \ar[r]^(.55){F} & \theta \ar@{^{(}->}[d]^{i}\\
W \ar[r]^{F} & P}
\end{equation}
\par
A similar construction works for $\AbFgGp$, the category
of finitely generated abelian groups. That is, we can construct a
category $\SubAbGp$ fibered in lattices over
$\AbFgGp$ by taking injective morphisms of
finitely generated abelian groups.

\subsection{An equivalence of fibered categories}
\label{subsec:equiv fibcat}

We now construct an explicit (contravariant) equivalence
between these two fibered categories considered above.
For any inclusion of subgroups $\eta \inj \xi$, let
\begin{equation}
\label{eq:ve2}
V(\eta \inj \xi):=(\Spm (\C[\xi/\eta])\inj\Spm(\C[\xi])).
\end{equation}
This algebraic subgroup is naturally identified with the subgroup of
characters of $\xi/\eta$, i.e., the closed embedding of subtori
$\Hom(\xi/\eta,\C^*)\to\Hom(\xi,\C^*)$.
\par
Finally, for any algebraic subgroups $W\inj P$, let
\begin{equation}
\label{eq:ev2} \epsilon(W\inj P):=(\{ \lambda \in \Hom(P,\C^*) :
\text{$t^{\lambda}-1$ vanishes on $W$}\}\subseteq \Hom(P,\C^*)).
\end{equation}
\par
For a fixed algebraic group $T$ in $\AbRed$,
with character group $H=\check{T}$, the subgroup
$\epsilon(W\inj T)$ of $H$ coincides with the subgroup
$\epsilon(W)$ defined in (\ref{eq:ev}), and the algebraic
subgroup $V(\xi\inj H)$ of $T$ coincides with the algebraic
subgroup $V(\xi)$ defined in (\ref{eq:ve}).
\par
Tracing through the definitions, we obtain the following
result, which reformulates Theorem~\ref{theorem-isomorphic}
in this setting.

\begin{theorem}
\label{thm: equivalence}
The two fibered categories $\SubAbRed$ and $\SubAbGp$ are
equivalent, with (contravariant) equivalences given by the functors
$V$ and $\epsilon$ defined above.
\end{theorem}

\section{Intersections of translated algebraic subgroups}
\label{eq:int tt}

In Sections \ref{sect:corresp} and \ref{sect:primitive},
we only considered intersections of
algebraic subgroups.  In this section, we consider the
more general  situation where translated subgroups intersect.

\subsection{Two morphisms}
\label{subsec:sigma gamma}
As usual, let $T$ be a complex abelian reductive group, and let
$H=\check{T}$ be the weight group corresponding to
$T=\widehat{H}$.

Let $\xi_1, \dots, \xi_k$ be subgroups
of $H$.  Set $\xi =\xi_1 + \cdots + \xi_k$, and let
$\sigma\colon \xi_1\times \cdots \times \xi_k \to
\xi$ be the homomorphism given by $(\lambda_1, \dots,
\lambda_k) \mapsto \lambda_1+ \cdots + \lambda_k$.
Consider the induced morphism on character groups,
\begin{equation}
\label{eq:sigmahat}
\xymatrix{\hat\sigma\colon \widehat{\xi} \ar[r] &
\widehat{\xi_1}\times \cdots \times \widehat{\xi_k}}.
\end{equation}
Using Remark \ref{rem:prods}, the next lemma is readily verified.

\begin{lemma}
\label{lem:sigma}
Under the isomorphisms $\widehat{\xi_i}\cong T/V(\xi_i)$ and
$\widehat{\xi}\cong T/V(\xi)=T/(\bigcap_i V(\xi_i))$, the
morphism $\hat\sigma$ gets identified with the morphism
$\delta\colon T/(\bigcap_i  V(\xi_i)) \to T/V(\xi_1) \times
\cdots \times T/V(\xi_k)$ induced by the diagonal map
$\Delta\colon T\to T^k$.
\end{lemma}

Next, let $\gamma\colon \xi_1\times \cdots \times \xi_k \to
H\times \cdots \times H$ be the product of the inclusion maps
$\gamma_i\colon \xi_i \inj H$, and consider the induced
homomorphism on character groups,
\begin{equation}
\label{eq:etasigma}
\xymatrix{\hat\gamma\colon \widehat{H}\times
\cdots \times \widehat{H} \ar[r]  &
\widehat{\xi_1}\times \cdots \times \widehat{\xi_k}}.
\end{equation}
The next lemma is immediate.

\begin{lemma}
\label{lem:gamma}
Under the isomorphisms $\widehat{\xi_i}\cong T/V(\xi_i)$, the
morphism $\hat\gamma$ gets identified with the projection map
$\pi\colon T^k \to T/V(\xi_1) \times \cdots \times T/V(\xi_k)$.
\end{lemma}

\subsection{Translated algebraic subgroups}
\label{subsec:translated}

Given an algebraic subgroup $P \subseteq T$, and an element
$\eta\in T$, denote by $\eta P$ the translate of $P$ by
$\eta$.  In particular, if $C$ is an algebraic subtorus of
$T$, then $\eta C$ is a translated subtorus. If $\eta$ is a
torsion element of $T$, we denote its order by
$\ord(\eta)$. Finally, if $A$ is a finite group, denote by
$c(A)$ the largest order of any element in $A$.

We are now in a position to state and prove the main result
of this section (Theorem \ref{thm:intro1} from the Introduction). 
As before, let $\xi_1,\dots , \xi_k$ be subgroups of $H=\check{T}$.  
Let $\eta_1, \dots, \eta_k$ be elements in $T$,
and consider the translated subgroups $Q_1, \dots,  Q_k$ of
$T$ defined by
\begin{equation}
\label{eq:qi} Q_i=\eta_i V(\xi_i).
\end{equation}
Clearly, each $Q_i$ is a subvariety of $T$, but, unless
$\eta_i\in V(\xi_i)$, it is not an algebraic subgroup.

\begin{theorem}
\label{translated-prop}
Set $\xi=\xi_1+ \cdots +\xi_k$ and $\eta = (\eta_1, \dots, \eta_k)\in
T^{k}$. Then
\begin{equation}
\label{eq:q1capq2}
Q_1\cap \cdots \cap Q_k =
\begin{cases}
\: \emptyset & \text{if $\hat\gamma(\eta) \notin \im(\hat\sigma)$}, \\[2pt]
\: \rho V(\xi) &\text{otherwise},
 \end{cases}
\end{equation}
where $\rho$ is any element in the intersection $Q=Q_1\cap \cdots \cap Q_k$.
Furthermore, if the intersection is non-empty, then
\begin{enumerate}
\item  \label{r1}
The variety $Q$ decomposes into
irreducible components as
$Q= \bigcup_{\tau \in \widehat{\overline{\xi}/\xi}}
\rho \tau V(\overline{\xi})$, and $\dim(Q)=\dim(T)-\rank(\xi)$.

\item \label{r2}
If $\eta$ has finite order, then $\rho$ can be chosen
to have finite order, too.  Moreover, $\ord(\eta) \mid
\ord(\rho) \mid \ord(\eta) \cdot c(\overline{\xi}/\xi)$.
\end{enumerate}
\end{theorem}

\begin{proof}
We have:
\begin{align*}
Q_1\cap \cdots \cap Q_k \neq \emptyset
&\same
\text{$\eta_1 a_1= \cdots =\eta_k a_k$, for some $a_i \in V(\xi_i)$}
\\
&\same
\hat\gamma_1(\eta_1)= \cdots = \hat\gamma_k(\eta_k)
& \text{by Lemma \ref{lem:gamma}\phantom{.}}
\\
&\same
\hat\gamma(\eta) \in \im(\hat\sigma)
& \text{by Lemma \ref{lem:sigma}.}
\end{align*}

Now suppose $Q_1\cap \cdots \cap Q_k \neq \emptyset$.
For any $\rho \in Q_1\cap \cdots \cap Q_k$, and any $1\le i\le k$,
there is a $\rho_i\in V(\xi_i)$ such that $\rho=\eta_i \rho_i$;
thus, $\rho^{-1} \eta_i \in V(\xi_i)$.
Therefore,
\begin{align*}
  \rho^{-1}(Q_1\cap \cdots \cap Q_k)
  & =\rho^{-1}(\eta_1 V(\xi_1) \cap \cdots \cap \eta_k V(\xi_k))\\
  &= \rho^{-1}(\eta_1 V(\xi_1)) \cap \cdots \cap \rho^{-1}(\eta_k V(\xi_k)) \\
  & = V(\xi_1) \cap \cdots \cap V(\xi_k)\\
  & = V(\xi_1 + \cdots + \xi_k).
\end{align*}
Hence, $Q_1\cap \cdots \cap Q_k = \rho V(\xi)$.

Finally, suppose $\eta$ has finite order. Let $\bar\rho$
be an element in $T/V(\xi)$ such that
$\hat\sigma(\bar\rho) = \hat\gamma(\eta)$.
Note that $\ord(\bar\rho)=\ord(\eta)$.
Using the exact sequence \eqref{eq:ex2}
and the third isomorphism theorem for
groups, we get a short exact sequence,
\begin{equation}
\label{eq:ex3}
\xymatrix{0 \ar[r] & \widehat{\overline{\xi}/\xi}   \ar[r]
& T/V(\overline{\xi}) \ar^{q}[r]&  T/V(\xi)    \ar[r]& 0}.
\end{equation}
Applying the $\Hom_{\group}(-,\C^*)$ functor to the split
exact \eqref{eq:split}, we get a split exact sequence,
\begin{equation}
\label{eq:ex4}
\xymatrix{0 \ar[r] & V(\overline{\xi})   \ar[r]
&  T \ar[r]&  T/V(\overline{\xi})
\ar@/_15pt/@{-->}_{s}[l]  \ar[r]& 0}.
\end{equation}

Now pick an element $\tilde\rho \in q^{-1}(\bar\rho)$.
We then have $q\big(\tilde\rho^{\,\ord(\eta)}\big)=\bar\rho^{\,\ord(\eta)}=1$,
which implies that $\tilde\rho^{\,\ord(\eta)}\in  \widehat{\overline{\xi}/\xi}$.
Hence, $\tilde\rho$ has finite order in $T/V(\overline{\xi})$, and, moreover,
$\ord(\eta) \mid \ord(\tilde\rho) \mid \ord(\tilde\rho) \cdot c(\overline{\xi}/\xi)$.
Setting $\rho = s(\tilde\rho)$ gives the desired translation factor.
\end{proof}

When $k=2$, the theorem takes a slightly simpler form.

\begin{corollary}
\label{cor:k=2}
Let $\xi_1$ and $\xi_2$ be two subgroups of $H$,
and let $\eta_1$ and $\eta_2$ be two elements in
$T=\widehat{H}$.  Then
\begin{enumerate}
\item \label{tt1}
The variety $Q=\eta_1 V(\xi_1) \cap \eta_2 V(\xi_2)$ is non-empty if and
only if $\eta_1 \eta_2^{-1}$ belongs to the subgroup $V(\xi_1)\cdot V(\xi_2)$.

\item \label{tt2}
If the above condition is satisfied, then $\dim Q=\rank H - \rank (\xi_1+\xi_2)$.
\end{enumerate}
\end{corollary}

In the special case when $H= \Z^r$ and $T=(\C^{*})^{r}$,
Theorem \ref{translated-prop} allows us to
recover Proposition 3.6 from \cite{Hi}.

\begin{corollary}[Hironaka \cite{Hi}]
\label{thm:hironaka}
Let $\xi_1,\dots , \xi_k$ be subgroups of $\Z^{r}$, 
let $\eta = (\eta_1, \dots, \eta_k)$ be an element in 
$(\C^*)^{rk}$, and set $Q_i=\eta_i V(\xi_i)$.   Then
\begin{equation}
\label{eq:q1qkhir}
Q_1\cap \cdots \cap Q_k \ne\emptyset  \same
\hat\gamma(\eta) \in \im(\hat\sigma).
\end{equation}
Moreover, for any connected component
$Q$ of $Q_1\cap \cdots \cap Q_k$, we have:
\begin{enumerate}
\item\label{h1}
$Q=\rho V(\overline{\xi})$, for some $\rho\in (\C^*)^r$.
\item\label{h2}
$\dim(Q) = r-\rank (\overline{\xi})$.
  \item\label{h3}
If $\eta$ has finite order, then $\ord(\eta) \mid
\ord(\rho) \mid \ord(\eta) \cdot c(\overline{\xi}/\xi)$.
\end{enumerate}
\end{corollary}

\subsection{Some consequences}
\label{subsec:apps}

We now derive a number of corollaries to Theorem \ref{translated-prop}.
Fix a complex abelian reductive group $T$.
To start with, we give a general description of the
intersection of two arbitrary unions of translated subgroups.

\begin{corollary}
\label{cor:arbint}
Let $W=\bigcup_i \eta_i V(\xi_i)$ and $W'=\bigcup_j \eta_j' V(\xi_j')$
be two unions of translated subgroups of $T$.  Then
\begin{equation}
\label{eq:cap}
W \cap W' = \bigcup_{i, j} \eta_i V(\xi_i) \cap \eta_j' V(\xi_j'),
\end{equation}
where $\eta_i V(\xi_i) \cap \eta_j' V(\xi_j')$ is either empty
(which this occurs precisely when
$\hat\gamma_{i, j}(\eta_i, \eta_j)$ does not belong to $\im(\hat\sigma_{i, j})$),
or equals $\eta_{i, j} V(\xi_i+\xi_j')$, for some $\eta_{i, j} \in T$.
\end{corollary}

\begin{corollary}
\label{cor:finite int}
With notation as above, $W \cap W'$ is finite if and only if
$W \cap W' = \emptyset$ or
$\rank (\xi_i + \xi_j')=\dim(T)$, for all $i, j$.
\end{corollary}

\begin{corollary}
\label{cor:finite int tors}
Let $W$ and $W'$
be two unions of (torsion-\nobreak) translated subgroups of $T$.  Then
$W\cap W'$ is again a union of (torsion-) translated subgroups of $T$.
\end{corollary}

The next two corollaries give a comparison between
the intersections of various translates of two fixed
algebraic subgroups of $T$.
Both of these results will be useful in another paper \cite{SYZ}.

\begin{corollary}
\label{dimension equal}
Let $T_1$ and $T_2$ be two algebraic subgroups in
$T$. Suppose $\alpha, \beta, \eta$ are elements in
$T$, such that $\alpha T_1 \cap \eta T_2 \neq \emptyset$ and
$\beta T_1\cap \eta T_2 \neq \emptyset$. Then
\begin{equation}
\label{eq:dim}
\dim\, (\alpha T_1 \cap \eta T_2)=\dim\, (\beta T_1 \cap \eta T_2).
\end{equation}
\end{corollary}

\begin{proof}
Set $\xi = \epsilon(T_1)  + \epsilon(T_2)$. From
Theorem \ref{translated-prop}, we find that both $\alpha T_1
\cap \eta T_2$ and $\beta T_1 \cap \eta T_2$ have dimension
equal to the corank of $\xi$.  This ends the proof.
\end{proof}

\begin{corollary}
\label{intersection empty}
Let $T_1$ and $T_2$ be two algebraic
subgroups in $T$. Suppose $\alpha_1$ and $\alpha_2$ are torsion
elements in $T$, of coprime order. Then
\begin{equation}
\label{eq:alpha}
T_1 \cap \alpha_2 T_2 = \emptyset \implies \alpha_1 T_1\cap \alpha_2
T_2 = \emptyset.
\end{equation}
\end{corollary}

\begin{proof}
Set $H=\check{T}$, $\xi_i=\epsilon( T_i)$, and
$\xi=\xi_1+\xi_2$. By Theorem \ref{translated-prop}, the
condition that $T_1 \cap \alpha_2 T_2 = \emptyset$ implies
$\hat\gamma(1, \alpha_2) \notin \im(\hat\sigma)$, where
$\sigma\colon \xi_1\times \xi_2 \surj \xi$ is the sum
homomorphism, and $\gamma\colon \xi_1\times \xi_2\inj
H\times H$ is the inclusion map.

Suppose $\alpha_1 T_1 \cap \alpha_2 T_2 \neq \emptyset$. Then, from
Theorem \ref{translated-prop} again, we know that $\hat\gamma(\alpha_1,
\alpha_2) \in \im(\hat\sigma)$; thus,
 $(\hat\gamma(\alpha_1, \alpha_2))^n \in \im(\hat\sigma)$, for any
 integer $n$.
From our hypothesis, the orders of $\alpha_1$ and $\alpha_2$ are
coprime;  thus, there exist integers $p$ and $q$ such that
$p\ord(\alpha_1)+q\ord(\alpha_2)=1$. Hence,
\[
\hat\gamma(1, \alpha_2)= \hat\gamma(\alpha_1^{p\ord(\alpha_1)},
\alpha_2^{1-q\ord(\alpha_2)})= (\hat\gamma(\alpha_1,
\alpha_2))^{p\ord(\alpha_1)}\in \im(\hat\sigma),
\]
a contradiction.
\end{proof}

\subsection{Abelian covers}
\label{subsec:abelcov}
In \cite{SYZ}, we use Corollaries~\ref{dimension equal}
and \ref{intersection empty} to study the homological
finiteness properties of abelian covers.  Let us briefly
mention one of the results we obtain as a consequence.

Let $X$ be a connected CW-complex with finite $1$-skeleton.
Let $H=H_1(X,\Z)$ the first homology group.
Since the space $X$ has only finitely many $1$-cells,
$H$ is a finitely generated abelian group. The characteristic varieties
of $X$ are certain Zariski closed subsets $\VV^i(X)$ inside the character
torus $\widehat{H}=\Hom(H,\C^*)$. The question we study in \cite{SYZ}
is the following:
Given a regular, free abelian cover $X^{\nu}\to X$, with
$\dim_{\Q} H_i(X^{\nu},\Q)<\infty$ for all $i\le k$,  which
regular, finite abelian covers of $X^\nu$ have the
same homological finiteness property?

\begin{prop}[\cite{SYZ}]
\label{prop:tt}
Let $A$ be a finite abelian group, of order $e$.
Suppose the characteristic variety $\VV^i(X)$ decomposes as
$\bigcup_{j} \rho_j T_j$, with each $T_j$ an algebraic subgroup of
$\widehat{H}$, and each $\rho_j$ an element in $\widehat{H}$ such that
$\bar{\rho_j}\in \widehat{H}/T_j$ satisfies $\gcd(\ord(\bar{\rho_j}),e)=1$.
Then, given any regular, free abelian cover $X^{\nu}$ with
finite Betti numbers up to some degree $k\ge 1$, the regular
$A$-covers of $X^{\nu}$ have the same finiteness property.
\end{prop}

\section{Exponential interpretation}
\label{sect:vexp}

In this section, we explore the relationship between the
correspondence $H \leadsto T=\widehat{H}$ from \S\ref{sect:corresp}
and the exponential map $\Lie(T) \to T$.

\subsection{The exponential map}
\label{subsec:exp}
Let $T$ be a complex abelian reductive group.
Denote by $\Lie(T)$ the Lie algebra of $T$.
The exponential map $\exp\colon \Lie(T) \to T$
is an analytic map, whose image is $T_0$,
the identity component of $T$.

As usual, set  $H=\check{T}$, and consider the lattice
$\HH=H^{\vee}\cong H/\Tors(H)$. We then can identify
 $T_0= \Hom(\HH^{\vee}, \C^*)$ and
$\Lie(T) =\Hom (\HH^{\vee},\C)$.  Under these
identifications, the corestriction to the image of
the exponential map can be written as
\begin{equation}
\label{eq:exp}
\exp=\Hom(-, e^{2 \pi \ii z})\colon\Hom(\HH^{\vee}, \C) \to
\Hom(\HH^{\vee}, \C^*),
\end{equation}
where $\C\to \C^*$, $z\mapsto e^z$ is the usual complex exponential.
Finally, upon identifying $\Hom (\HH^{\vee},\C)$ with $\HH\otimes \C$,
we see that $T_0=\exp(\HH\otimes \C)$.

The correspondence $T\leadsto  \HH=(\check{T})^{\vee}$
sends an algebraic subgroup $W$ inside $T$ to the
sublattice $\chi=(\check{W})^{\vee}$
inside $\HH$.  Clearly, $\chi=\Lie(W) \cap \HH$  is a
primitive lattice; furthermore,
$\exp(\chi \otimes \C)=W_0$.

Now let $\chi_1$ and $\chi_2$ be two sublattices in $\HH$.
Since the exponential map is a group homomorphism,
we have the following equality:
\begin{equation}
\label{eq:exp2}
\exp((\chi_1+\chi_2) \otimes \C)=\exp(\chi_1 \otimes \C)\cdot
\exp(\chi_2 \otimes \C).
\end{equation}
On the other hand, the intersection of the two algebraic
subgroups $\exp(\chi_1 \otimes \C)$ and $\exp(\chi_2 \otimes \C)$
need not be connected, so it cannot be expressed solely in terms
of the exponential map.  Nevertheless, we will give a precise formula
for this intersection in Theorem \ref{thm:exp} below.

\subsection{Exponential map and Pontrjagin duality}
\label{subsec:exp pont}
First, we need to study the relationship between the exponential map
and the correspondence from  Proposition \ref{prop:abred}.

\begin{lemma}
\label{lem:vexp}
Let $T$ be a complex abelian reductive group, and let
$\HH=(\check{T})^{\vee}$.  Let $\chi\le \HH$ be a sublattice.
We then have an equality of connected algebraic subgroups,
\begin{equation}
\label{eq:chi}
V((\HH/\chi)^{\vee})= \exp(\chi \otimes \C),
\end{equation}
inside $T_0=\exp(\HH \otimes \C)$.
\end{lemma}

\begin{proof}
Let $\pi\colon \HH\to \HH/\chi$ be the canonical projection,
and let $K=\coker(\pi^{\vee})$.
As in \eqref{eq:K}, we have an exact sequence,
$0 \to  K \to \chi^{\vee} \to \overline{\chi}/\chi \to  0$.
Applying the functor $\Hom(-, \C^*)$ to this
sequence, we obtain a new short exact sequence,
\begin{equation}
\label{eq:hom} \xymatrix{ 0  \ar[r]& \Hom(\overline{\chi}/\chi,\C^*)
\ar[r]& \Hom(\chi^{\vee}, \C^*)  \ar[r]& \Hom(K, \C^*)  \ar[r]& 0 }.
\end{equation}
From the way the functor $V$ was defined in \eqref{eq:ve},
we have that $\Hom(K, \C^*) =V((\HH/\chi)^{\vee})$.
Composing with the map $\exp\colon
\C\to \C^*$, we obtain the following commutative diagram:
\begin{equation}
\label{eq:homdiag}
\xymatrixrowsep{18pt}
\xymatrix{ \Hom(\chi^{\vee}, \C)
\ar[r]^(.48){\exp} \ar[d]^{\cong}
            & \Hom(\chi^{\vee}, \C^*) \ar@{->>}[d]\\
            \Hom(K, \C) \ar@{^{(}->}[d] \ar[r]^(.36){\exp}
            & \Hom(K, \C^*)=V((\HH/\chi)^{\vee}) \ar@{^{(}->}[d]\\
             \Hom(\HH^{\vee}, \C) \ar[r]^(.45){\exp}& \Hom(\HH^{\vee},\C^*)
             }
\end{equation}

Identify now $\exp(\chi \otimes \C)$ with the image of
$\Hom(\chi^{\vee}, \C)$ in $\Hom(\HH^{\vee}, \C^*)$. Clearly,
this image  coincides with the image of $V((\HH/\chi)^{\vee})$
in $T_0=\Hom(\HH^{\vee},\C^*)$, and so we are done.
\end{proof}

\begin{corollary}
\label{cor:hexp}
Let $H$ be a finitely generated abelian group, and
let $\xi \le H$ be a subgroup.
Consider the sublattice $\chi=(H/\xi)^{\vee}$ inside $\HH=H^{\vee}$.
Then
\begin{equation}
\label{eq:hchi}
V(\overline{\xi})= \exp(\chi \otimes \C).
\end{equation}
\end{corollary}

\begin{proof}
Note that $\overline{\xi}=(\HH/\chi)^{\vee}$, as subgroups of
$H/\Tors(H)=\HH^{\vee}$.  The desired equality follows at
once from Lemma \ref{lem:vexp}.
\end{proof}

\subsection{Exponentials and determinant groups}
\label{subsec:exp det}

We are now in a position to state and prove the main result
of this section (Theorem \ref{thm:intro2} from the Introduction).

\begin{theorem}
\label{thm:exp}
Let $T$ be a complex abelian reductive group, and
let $\chi_1$ and $\chi_2$ be two sublattices of $\HH=\check{T}^{\vee}$.

\begin{enumerate}
\item \label{x1}
Set $\xi=(\HH/\chi_1)^{\vee} +(\HH/\chi_2)^{\vee}\le \HH^{\vee}$
and $\chi=(\HH^{\vee}/\overline{\xi})^{\vee}\le \HH$.  Then
\[
\exp(\chi_1 \otimes \C) \cap \exp(\chi_2 \otimes \C) =
\widehat{\overline{\xi}/\xi}\cdot \exp(\chi\otimes \C),
\]
as algebraic subgroups of $T_0$.
Moreover, the identity component of both these groups is
$V(\overline{\xi})=\exp(\chi \otimes \C)$.

\item \label{x2}
Now suppose  $\chi_1$ and $\chi_2$ are
primitive sublattices of $\HH$, with $\chi_1 \cap \chi_2=0$.
We then have
an isomorphism of finite abelian groups,
\[
\exp(\chi_1 \otimes \C) \cap \exp(\chi_2 \otimes \C) \cong
\overline{\chi_1 + \chi_2}/\chi_1 + \chi_2.
\]
\end{enumerate}
\end{theorem}

\begin{proof}
To prove part \eqref{x1}, note that
\begin{align*}
  \exp(\chi_1 \otimes \C) \cap \exp(\chi_2 \otimes \C)
  &=  V((\HH/\chi_1)^{\vee}) \cap V((\HH/\chi_2)^{\vee})
  & \text{by Lemma  \ref{lem:vexp}\phantom{.}} \\
  & = V((\HH/\chi_1)^{\vee} + (\HH/\chi_2)^{\vee})
  & \text{by Lemma \ref{lem:vlattice}\phantom{.}}\\
  & = \widehat{\overline{\xi}/\xi}\cdot V(\overline{\xi})
  & \text{by Lemma \ref{lem:vxi}.}
\end{align*}
Finally, note that $(\HH/\chi)^{\vee}=\overline{\xi}$; thus,
$V(\overline{\xi})=\exp(\chi\otimes \C)$, again
by Lemma  \ref{lem:vexp}.

To prove part \eqref{x2}, note that  $\overline{\xi}=\HH^{\vee}$,
since we are assuming $\chi_1\cap \chi_2=0$.
Hence, $V(\overline{\xi})=\{1\}$.   Since we are
also assuming that the lattices $\chi_1$ and
$\chi_2$ are primitive, Proposition \ref{prop:xi} applies,
giving that $\overline{\xi}/\xi \cong
\overline{\chi_1 + \chi_2}/\chi_1 + \chi_2$.
Using now part \eqref{x1} finishes the proof.
\end{proof}

The next corollary follows at once from Theorem \ref{thm:exp}.

\begin{corollary}
\label{cor:exph}
Let $H$ be a finitely generated abelian group, and
let $\xi_1$ and $\xi_2$ be two subgroups.
Let $\HH=H^{\vee}$ be the dual lattice, and
$\chi_i =(H/\xi_i)^{\vee}$ the corresponding sublattices.

\begin{enumerate}
\item \label{y1} Set $\xi=\overline{\xi_1} + \overline{\xi_2}$. Then:
\[
\exp(\chi_1 \otimes \C) \cap \exp(\chi_2 \otimes \C) =
\widehat{\overline{\xi}/\xi}\cdot  V(\overline{\xi}).
\]

\item \label{y2}
Now suppose $\rank(\xi_1 + \xi_2)=\rank(H)$.
Then
\[
\exp(\chi_1 \otimes \C) \cap \exp(\chi_2 \otimes \C) \cong
\overline{\chi_1 + \chi_2}/\chi_1 + \chi_2\cong
\overline{\xi}/\xi.
\]
\end{enumerate}
\end{corollary}

\subsection{Some applications}
\label{subsec:nazir}
Let us now consider the case when $T=(\C^*)^r$.
In this case, $\HH=\Z^r$, and the exponential map \eqref{eq:exp}
can be written in coordinates as $\exp\colon \C^r\to (\C^*)^r$,
$(z_1,\dots ,z_r)\mapsto (e^{2 \pi \ii z_1},\dots ,e^{2 \pi \ii z_r})$.
Applying Theorem \ref{thm:exp} to this situation,
we recover Theorem 1.1 from \cite{Na}.

\begin{corollary}[Nazir \cite{Na}]
\label{cor:nazir}
Suppose $\chi_1$ and $\chi_2$ are primitive lattices in $\Z^r$,
and $\chi_1 \cap \chi_2=0$. Then
$\exp(\chi_1 \otimes \C) \cap \exp(\chi_2 \otimes \C)$ is isomorphic to
$\overline{\chi_1 + \chi_2}/\chi_1 + \chi_2$.
\end{corollary}

As another application, let us show that intersections of
subtori can be at least as complicated as arbitrary finite
abelian groups.

\begin{corollary}
\label{cor:dimtori}
For any integer $n\ge 0$, and any finite abelian group $A$,
there exist subtori $T_1$ and $T_2$ in some complex
algebraic torus $(\C^{*})^r$ such that
$T_1 \cap T_2\cong (\C^*)^n \times A$.
\end{corollary}

\begin{proof}
Write $A= \Z_{d_1}\oplus\cdots\oplus\Z_{d_k}$.
Let $\chi_1$ and $\chi_2$ be the lattices in $\Z^{2k}$
spanned by the columns of the matrices
$\left(\begin{smallmatrix} I_k \\ 0 \end{smallmatrix}\right)$
and $\left(\begin{smallmatrix} I_k\\ D \end{smallmatrix}\right)$,
where $D=\diag(d_1,\dots,d_k)$.   Clearly, both $\chi_1$ and
$\chi_2$ are primitive, and $\chi_1\cap \chi_2=0$,
yet the lattice $\chi=\chi_1+\chi_2$ is {\em not}\/ primitive
if the group $A=\overline{\chi}/\chi$ is non-trivial.

Now consider the subtori $P_i=\exp(\chi_i \otimes \C)$ in $(\C^*)^{2k}$,
for $i=1,2$.   By Corollary \ref{cor:nazir}, we have that
$P_1 \cap P_2 \cong A$. Finally, consider the subtori
$T_i=(\C^*)^n  \times P_i$ in $(\C^*)^{n+2k}$. Clearly,
$T_1 \cap T_2\cong (\C^*)^n \times A$, and we are done.
\end{proof}

In particular, if $A$ is a finite cyclic group, there exist
$1$-dimensional subtori, $T_1$ and $T_2$, in
$(\C^{*})^{2}$ such that $T_1\cap T_2\cong A$.

\section{Intersections of torsion-translated subtori}
\label{sect:tors}

In this section, we revisit Theorem \ref{translated-prop} from
the exponential point of view.
In the case when the translation factors have finite order,
the criterion from that theorem can be refined,
to take into account certain arithmetic information about the
translated tori in question.

\subsection{Virtual belonging}
\label{subsec:virtual}

We start with a definition.

\begin{definition}
\label{def:virtual}
Let $\chi$ be a primitive lattice in $\Z^r$.  Given a vector
$\lambda\in \Q^r$, we say $\lambda$
\em{virtually belongs to $\chi$}\/ if
$d\cdot \lambda\in \Z^r$, where $d=\abs{\det\, [\chi \mid \chi_0]}$
and $\chi_0=\Q\lambda \cap \Z^r$.
\end{definition}

Here, $ [\chi \mid \chi_0]$ is the matrix obtained by concatenating
basis vectors for the sublattices $\chi$ and $\chi_0$ of $\Z^r$.
Note that $\chi_0$ is a (cyclic) primitive lattice, generated by
an element of the form $\lambda_0=m \lambda  \in \Z^r$, for
some $m\in \N$.

In the next lemma, we record several properties of the notion
introduced above.

\begin{lemma}
\label{lem:virtual}
With notation as above,
\begin{enumerate}
\item \label{o0}
$\lambda$ virtually belongs to $\chi$ if and only if
$d \lambda\in \chi_0$.
\item \label{o1}
$d=0$ if and only if $\lambda\in \chi\otimes \Q$, which in turn implies that
$\lambda$ virtually belongs to $\chi$.
\item \label{o2}
If $d>0$, then $d$ equals the order
of the determinant group $\overline{\chi+\chi_0}/\chi+\chi_0$.
\item \label{o3}
If $d=1$, then $\lambda$ virtually belongs to $\chi$
if and only if $\lambda\in \Z^r$, which happens if and only if
$\lambda \in \chi_0$.
\end{enumerate}
\end{lemma}

Next, we give a procedure for deciding when a torsion
element in a complex algebraic torus belongs to a subtorus.

\begin{lemma}
\label{lem:peta}
Let $P\subset (\C^*)^r$ be a subtorus, and
let $\eta\in (\C^*)^r$ be an element of finite order.
Write $P=\exp(\chi\otimes \C)$, where $\chi$ is a primitive lattice in $\Z^r$.
Then $\eta \in P$ if and only if $\eta=\exp(2\pi \ii \lambda)$, for
some $\lambda\in \Q^r$ which virtually belongs to $\chi$.
\end{lemma}

\begin{proof}
Set $n=\rank \chi$, and write $\eta=\exp(2\pi \ii \lambda)$, for some
$\lambda=(\lambda_1,\dots, \lambda_r)\in \Q^r$. Let $\lambda_0$
be a generator of
$\chi_0=\Q\lambda \cap \Z^r$, and write $\lambda = q \lambda_0$.
Finally, set $d=\abs{\det\, [\chi \mid \chi_0]}$.

First suppose $d=0$. Then, by Lemma \ref{lem:virtual}\eqref{o1},
$\lambda$ virtually belongs to $\chi$.  Since $\lambda\in \chi\otimes \Q$,
we also have $\eta\in P$, and the claim is established in this case.

Now suppose $d\neq 0$.  As in the proof of Proposition~\ref{prop:xi},
we can choose a basis for $\Z^r$ so that the inclusion of $\chi+\chi_0$
into $\overline{\chi+\chi_0}$ has matrix of the form \eqref{eq:mat1},
with $D=d$.  In this basis, the lattice $\chi$ is the span of the
first $n$ coordinates, and $\chi_0$ lies in the span of the first
$n+1$ coordinates.  Also in these coordinates,
$\lambda_{n+1}=q$, and so $\eta_{n+1}=e^{2\pi \ii q}$;
furthermore, $\eta_{n+2}=\cdots =\eta_r=1$.
On the other hand, $P=\{ z \in (\C^*)^r \mid z_{n+1}=\dots =z_r=1\}$.
Therefore, $\eta\in P$ if and only if $d q$ is an integer, which
is equivalent to $d\lambda\in \chi_0$.
\end{proof}

\subsection{Torsion-translated tori}
\label{subsec:tt}

We are now ready to state and prove the main results of this section
(Theorem \ref{thm:intro3} from the Introduction).

\begin{theorem}
\label{thm:k=2}
Let $\xi_1$ and $\xi_2$ be two sublattices in $\Z^r$.
Set $\varepsilon=\xi_1\cap \xi_2$, and write
$\widehat{\overline{\varepsilon}/\varepsilon}=\{\exp(2\pi \ii \mu_k) \}_{k=1}^s$.
Also let $\eta_1$ and $\eta_2$ be two torsion elements in $(\C^*)^r$,
and write $\eta_j=\exp(2\pi \ii \lambda_j)$. The following are equivalent:
\begin{enumerate}
\item  \label{q1}
The variety $Q=\eta_1 V(\xi_1) \cap \eta_2 V(\xi_2)$ is non-empty.
\item \label{q2}
One of the vectors $\lambda_1-\lambda_2 - \mu_k$ virtually belongs to
the lattice $(\Z^r/\varepsilon)^{\vee}$.
\end{enumerate}
If either condition is satisfied, then $Q=\rho V(\xi_1+\xi_2)$,
for some $\rho \in Q$.
\end{theorem}

\begin{proof}
Follows from Corollary \ref{cor:k=2} and Lemma \ref{lem:peta}.
\end{proof}

This theorem provides an efficient algorithm for checking whether two
torsion-trans\-lated tori intersect.  We conclude with an example
illustrating this algorithm.

\begin{example}
\label{ex:alg}
Fix the standard basis $e_1,e_2,e_3$ for $\Z^3$, and consider
the primitive sublattices $\xi_1=\spn(e_1,e_2)$ and $\xi_2=\spn(e_1,e_3)$.
Then $\varepsilon=\spn(e_1)$ is also primitive, and so $s=1$ and $\mu_1=0$.
For selected values of $\eta_1, \eta_2\in (\C^*)^3$, let us decide
whether the set $Q=\eta_1 V(\xi_1) \cap \eta_2 V(\xi_2)$
is empty or not, using the above theorem.

First take $\eta_1=(1, 1, 1)$ and $\eta_2=(1, e^{2\pi \ii /3},1)$,
and pick $\lambda_1=0$ and $\lambda_2=(1,\frac{1}{3},0)$.
One easily sees that $d=\det\left(\begin{smallmatrix}
 0 & 0 &3\\
 0 & 1&1\\
 1 & 0&0\\
\end{smallmatrix}\right)=-3$,
and $d(\lambda_1-\lambda_2-\mu_1)=(3, 1, 0)\in \Z^3$.
Thus, $Q \neq \emptyset$.

Next, take the same $\eta_1$ and $\lambda_1$, but
take $\eta_2=(e^{3\pi \ii/2}, 1, 1)$ and $\lambda_2=(\frac{3}{4}, 0,1)$.
In this case, $d=\det\left(\begin{smallmatrix}
0 & 0&3\\
0 & 1&0\\
1 & 0&4\\
\end{smallmatrix}\right)=-3$ and $d(\lambda_1-\lambda_2-\mu_1)=
(\frac{9}{4}, 0, 3)\notin \Z^3$. Thus, $Q = \emptyset$.
\end{example}


\newcommand{\arxiv}[1]
{\texttt{\href{http://arxiv.org/abs/#1}{arXiv:#1}}}
\newcommand{\doi}[1]
{\texttt{\href{http://dx.doi.org/#1}{doi:#1}}}
\renewcommand{\MR}[1]
{\href{http://www.ams.org/mathscinet-getitem?mr=#1}{MR#1}}

\end{document}